\documentclass[a4paper,reqno,11pt]{amsart}

\usepackage[all]{xy}
 \usepackage[utf8x]{inputenc}
\usepackage[english]{babel}
\usepackage{amssymb}
\usepackage{graphicx}
\usepackage{hhline}
\usepackage{xcolor}
\usepackage{wasysym}
\usepackage{tikz}
\usepackage{amsmath,amsthm,amssymb,xypic,verbatim}
\usepackage{url}

\usepackage{geometry}
\newcommand{\p}{\mathbb{P}}

\newcommand{\Z}{\mathbb{Z}}
\newcommand{\Q}{\mathbb{Q}}
\newcommand{\C}{\mathbb{C}}
\newcommand{\R}{\mathbb{R}}
\newcommand{\E}{\mathcal{E}}

\newcommand{\vu}{\varnothing}

\newcommand{\AAA}{\mathcal{A}}
\newcommand{\rk}{\mathop{\rm rk}\nolimits}

\newcommand{\vol}{\mathop{\rm vol}\nolimits}

\newcommand{\Span}[1]{\langle#1\rangle}
\newcommand{\ov}[1]{\overline{#1}}

\newcommand{\GL}{\mathop{\rm GL}\nolimits}

\DeclareMathOperator{\diag}{diag}
\DeclareMathOperator{\Id}{Id}
\DeclareMathOperator{\irr}{irr}

\theoremstyle{plain}

\newtheorem{thm}{Theorem}
\newtheorem{pro}[thm]{Proposition}
\newtheorem{lem}[thm]{Lemma}   
\newtheorem{cor}[thm]{Corollary}

\theoremstyle{definition}

\newtheorem{notaz}[thm]{Notation}

\newtheorem{es}[thm]{Example}

\newtheorem{rmk}[thm]{Remark}

\newcommand{\Mat}[0]{\operatorname{Mat}}
\newcommand{\graffe}[1]{\left\lbrace #1\right\rbrace }
\newcommand{\mc}[1]{\mathcal{#1}}
\newcommand{\gen}[1]{\langle #1\rangle}

\title{Toric varieties from cyclic matrix semigroups}
\author{Francesco Galuppi}
\address{Department of Mathematics and Geoscience, University of Trieste, Via Weiss 2, 34128 Trieste, Italy}
\email{fgaluppi@units.it}
\author{Mima Stanojkovski}
\address{Max Planck Institute for Mathematics in the Sciences\\
	Inselstra\ss{}e 22\\
	04103 Leipzig\\
	Germany}
\email{mima.stanojkovski@mis.mpg.de}
\thanks{We are very thankful to Bernd Sturmfels for suggesting this project. We also thank Paul G\"orlach, Lorenzo Venturello and Stefano Marseglia for mathematical discussions, and Alastair Litterick for feedback on an early version of this manuscript.
}

\begin{document}
	\maketitle

\begin{abstract}
We present and expand some existing results on the Zariski closure of cyclic groups and semigroups of matrices. We show that, with the exclusion of isolated points, their irreducible components are toric varieties. Additionally, we demonstrate how every toric variety can be realized as the Zariski closure of a cyclic matrix group. Our paper includes a number of explicit examples and a note on existing computational results.
\end{abstract}

\section*{Introduction}\label{sec:intro}
In mathematics, as well as in many applied sciences, researchers often face the problem of describing a complicated behaviour or a sophisticated model. A common approach is to find \textit{invariants}: roughly speaking, an invariant is a property shared by every point of the model or a function that attains the same value at every state.
%
Invariants appear in a wide range of areas of mathematics, physics, and computer science.
As an example, in the study of dynamical systems, invariants can determine whether the system will reach a given state. 

From an algebraic viewpoint, the most meaningful invariants are polynomial functions. To compute the polynomials that vanish on a given model or set means to compute the closure of such set in the Zariski topology. A common approach in applied algebraic geometry is to take a model coming from biology, statistics or computer science, and give it the structure of an algebraic variety, thus allowing the use of powerful geometric techniques. On the other hand, these classes of models provide examples of families of varieties, whose geometry is interesting in their own right.

In this paper we are interested in the Zariski closures of subsemigroups of $\Mat_n(\C)$: given a semigroup $X\subseteq\Mat_n(\C)$, its Zariski closure is the smallest algebraic subvariety of $\Mat_n(\C)$ containing $X$. 
When $X$ is a closed subset of $\GL_n(\C)$ in the induced topology, one calls $X$ an algebraic group. 
 The study of algebraic groups 
 has a long history and a rich literature (important references are \cite{Humphreys
	,Springer}), but it is also motivated by concrete applications. For instance, groups generated by matrices appear naturally in dynamical systems, where they are often called \textit{automata} or \textit{affine programs} (see for example \cite[Section 1]{HOPW/18}). 

From a computational viewpoint, the problem becomes to find an algorithm that, given a finite set of matrices, returns the Zariski closure of the group or the semigroup that they generate, see for instance \cite[Chapter 4]{deGraaf}. Some of the results in the literature address the existence of an algorithm rather than its implementation or the geometric properties of the closure, see \cite[Theorem 9]{DJK} and \cite[Theorem 16]{HOPW/18}. 

The aim of this article is twofold. On the one hand, we present new proofs of some known results and use geometric techniques to expand and generalize them. On the other hand, we hope that this paper will serve as a clear, accessible reference for researchers working in different areas of mathematics and computer science, as well as 
a friendly entrance point for those who are interested in the subject.

It is natural to start with the simplest situation, i.e.\ the closure of a cyclic group or semigroup. In this case we are able to give a detailed description of the closure: what strikes us as remarkable is that each irreducible component turns out to be a \textit{toric variety}. 
Roughly speaking, a variety is toric if it is the image of a monomial map. A toric variety not only has very pleasant properties - to name a few, it is irreducible, rational and its ideal is generated by binomials - but it can also be associated to a polytope that completely encodes its geometry. This makes toric varieties accessible from a theoretical, combinatorial, and computational point of view.  For instance, there are effective techniques to determine their degrees and their equations. For more information on toric varieties we refer to \cite{Cox}. 
We conclude by pointing out that binomial ideals themselves sit in a very fertile ground between geometry, algebra, and combinatorics 
\cite{ES}. 

\begin{notaz}\label{not}Here we fix the notation we use in this paper.
	\begin{enumerate}
		\item For a subset $X$ of $\Mat_n(\C)$, we denote by $\overline{X}$ the Zariski closure of $X$ in $\Mat_n(\C)$, regarded as $\C^{n^2}$. We write 
	$\irr(\overline{X})$ 
	for the 
	number of irreducible components of $\overline{X}$. 
		\item Given a matrix $M\in\Mat_n(\C)$, we denote by $\mc{E}(M)$ the set of nonzero eigenvalues of $M$. If $\mc{E}(M)\neq\vu$, then we write $G(M)$ for the multiplicative subgroup of $\C^*$ generated by $\mc{E}(M)$.
	
\item For a finitely generated abelian group $G$, i.e.\ a finitely generated $\Z$-module, we write $G_{\mathrm{tor}}$ for the torsion submodule of $G$ and $\rk G$ for the rank of a free complement of $G_{\mathrm{tor}}$ in $G$. For a finite group $G$, we denote by $|G|$ its order. \end{enumerate}\end{notaz}



\begin{thm}\label{thm:semigroups}
	Let $M\in\Mat_n(\C)$ be a nonzero matrix and let $\nu
	$ 
	be the largest size of a Jordan block of $M$ associated to $0$. Write $X=\{M^k\mid k\in\Z_{>0}\}$ for the semigroup of $\Mat_n(\C)$ generated by $M$. 
	Then $\overline{X}$ can be written as a disjoint union
	\[
	\overline{X} = X_0\ \dot{\cup}\ X_1
	\]
of closed sets, where
\begin{enumerate}
	\item 
	$X_0$ is a collection of points of cardinality
	\[
	\begin{cases}
	\nu & \textup{if } \E(M)=\vu, \\
	\max\{0,\nu-1\} & \textup{otherwise}.
	\end{cases}
	\]
	\item either $X_1=\mc{E}(M)=\vu $ or 
	 $X_1$ is a union of $|G(M)_{\mathrm{tor}}|$  toric varieties of dimension 
	\[
	\dim X_1=\begin{cases}
	\rk G(M) & \textup{if $M^{\max\graffe{1,\nu}}$ is diagonalizable}, \\
	\rk G(M)+1 & \textup{otherwise}.
	\end{cases}
	\]
\end{enumerate}
\end{thm}

Observe that Theorem \ref{thm:semigroups} applies not only to semigroups: as we will prove in Proposition \ref{pro:semi vs gp}, when $M$ is invertible the same statement is true for the group generated by $M$. In this case, the toric varieties described in point (2) are the irreducible components of $X_1$, and their intersections with $\GL_n(\C)$ are the connected components of the group $\ov{\Span{M}}\cap\GL_n(\C)$.  
When $M$ is invertible and diagonalizable, Theorem \ref{thm:semigroups} agrees with \cite[Proposition 3.9.7]{deGraaf}. Let us also point out that, thanks to \cite[Proposition 11]{hybrid}, Theorem \ref{thm:semigroups} describes not only the structure of the closure of affine programs, which are discrete dynamical systems, but also the structure of a much larger class of dynamical systems, called hybrid automata.

\begin{es}\label{ex:joel}
	Let
	\[
	M=\begin{pmatrix}
	10 & -8 \\
	6 & -4
	\end{pmatrix}\in\GL_2(\C)
	\]
	and let $X$ be the semigroup of $\GL_2(\C)$ generated by $M$. If we set
	\[D=\begin{pmatrix}2 & 0 \\
	0 & 4\end{pmatrix}\mbox{ and } P=\begin{pmatrix}1 & 4 \\1 & 3\end{pmatrix},
	\]
	then $M=PDP^{-1}$.
	It follows that $M$ is diagonalizable, $\mc{E}(M)=\{2,4\}$, and 
	$G(M)=\gen{2,4}=\gen{2}\cong \Z$. Theorem \ref{thm:semigroups} yields that $\overline{X}$ is an irreducible toric curve in $\C^4$. This example was presented in \cite[Section 2]{HOPW/18} in the setting of dynamical systems. 
	Here we determine explicit equations describing the closure of $X$. Let $Y$ be the semigroup generated by $D$. Denoting the coordinates of $\C^4$ by
	\[\begin{pmatrix}
	x & w\\
	z & y
	\end{pmatrix},\]
	we see that the three polynomials $f=z$, $g=w$, and $h=x^2-y$ generate the ideal of $\overline{Y}$. Let $\phi:\C^4\to\C^4$ be the linear automorphism defined by
	\begin{align*}
	\begin{pmatrix}
	x & w\\
	z & y
	\end{pmatrix}&\longmapsto P^{-1}\begin{pmatrix}
	x & w\\
	z & y
	\end{pmatrix}P=\begin{pmatrix}
	-3x+4y+4z-3w & -12x+12y+16z-9w \\
	x-y-z+w & 4x-3y-4z+3w
	\end{pmatrix}.
	\end{align*}
	Then $\phi(\overline{X})=\overline{Y}$, hence $f\circ\phi$, $g\circ\phi$, and $h\circ\phi$ generate the ideal of $\ov{X}$. With this choice of coordinates, the map $\phi$ is represented by the matrix
	\[
	\begin{pmatrix}
	-3 & 4&4 &-3\\
	4 & -3 & 4 & 3\\
	1 & -1 & -1 & 1\\
	-12 & 12 & 16 & -9
	\end{pmatrix},
	\]
	therefore $\overline{X}$ is described by the equations
	\[\begin{cases}
	x + w = y +z,\\
	12x + 9w = 12y+16z,\\	
	(-3x+4y+4z-3w)^2=4x-3y-4z+3w
	.
	\end{cases}\]
	These 
	provide the tightest polynomial conditions that a point has to satisfy in order to belong to $X$.
\end{es}


\section{Preliminaries}\label{sec:basic}

In the present paper we are concerned with Zariski closures of subsets of $\Mat_n(\C)$. However, when the subsets in play consist of invertible matrices, the closures are classically taken in $\GL_n(\C)$. Here we show that, when dealing with commutative subgroups, some important geometric properties do not depend on this choice. 

\begin{lem}\label{lem:conjugation}
Let $X$ be a subgroup of $\GL_n(\C)$ and let $g\in\GL_n(\C)$. Then $\overline{gXg^{-1}}=g\overline{X}g^{-1}$ and $\overline{X}$ is isomorphic to $\overline{gXg^{-1}}$ as algebraic subvarieties of $\Mat_n(\C)$.
\end{lem}

\begin{proof}
Let $\phi:\Mat_n(\C)\rightarrow\Mat_n(\C)$ denote conjugation under $g$, which is a homeomorphism restricting to an automorphism of the algebraic group $\GL_n(\C)$. As a consequence, $\overline{X}$ and $\phi(\overline{X})=g\overline{X}g^{-1}$ are isomorphic varieties. The morphism $\phi$ being a homeomorphism, we get $\phi(\overline{X})=\overline{\phi(X)}$.
\end{proof}

\begin{lem}\label{lem:springer}
Let $X$ be a commutative subgroup of $\GL_n(\C)$. Then the following hold:
\begin{enumerate}
 \item There exists $g\in\GL_n(\C)$ such that $gXg^{-1}$ consists of upper triangular matrices.
 \item \label{item:density}The Zariski closure $\overline{X}\cap\GL_n(\C)$ of $X$ in $\GL_n(\C)$ is dense in $\overline{X}$.
\end{enumerate}
\begin{proof} For (1), see \cite[Lemma 2.4.2]{Springer}. To see (2), observe that $X\subseteq \overline{X}\cap \GL_n(\C) \subseteq \overline{X}$,
so taking the closures yields the claim.
\end{proof}
\end{lem}

\begin{rmk}\label{rmk:dim-irr}
We will assume in the rest of the paper that any commutative subgroup of $\GL_n(\C)$ is already given in upper triangular form
. Moreover, thanks to Lemma \ref{lem:springer}(\ref{item:density}), dimension and number of irreducible components of $\overline{X}$ are the same, regardless of whether we take them in $\Mat_n(\C)$ or $\GL_n(\C)$. 
\end{rmk}


Besides the choice of the ambient space for the closure, i.e.\ $\Mat_n(\C)$ or $\GL_n(\C)$, there are other variations of the problem in the literature. As we pointed out in the introduction, given finitely many matrices, it is interesting to consider both the group and the semigroup they generate. The following result, already proven in \cite[Lemma 2]{DJK} for orthogonal matrices, shows that, for our purposes, it is equivalent to deal with groups or semigroups.

\begin{pro}\label{pro:semi vs gp}
Let $Y \subseteq \GL_n(\C)$ 
be a subsemigroup and let $X$ denote the smallest subgroup of $\GL_n(\C)$ containing $Y$, i.e.\
\[
X = \bigcap\graffe{H\leq \GL_n(\C) \mid Y\subseteq H}.
\]
Then the Zariski closures $\overline{X}$ and $\overline{Y}$ are the same.
\end{pro}

\begin{proof}
Let $U_X=\overline{X}\cap\GL_n(\C)$ and $U_Y=\overline{Y}\cap \GL_n(\C)$ denote respectively the closures of $X$ and $Y$ in $\GL_n(\C)$. 
We claim that $U_Y$ is a subgroup of $U_X$.
Indeed, if this were not the case, there would exist an element $g\in U_Y$ yielding an infinite chain $U_Y\supsetneq gU_Y \supsetneq g^2U_Y\supsetneq \ldots$  and thus contradicting Noetherianity of the Zariski closure.
Since $U_X$ is the smallest closed subgroup of $\GL_n(\C)$ containing $X$, the equality $U_X=U_Y$ holds. We now observe that $X\subseteq U_X\subseteq \overline{X}$ and so $\ov{X}=\overline{U_X}$. An analogous statement holds for $U_Y$ and so we conclude that $\overline{X}=\overline{Y}$.
\end{proof}

\section{Zariski closure of a cyclic group}\label{sec:closure of subgroups}


In the present section, we will prove Theorem \ref{thm:semigroups} for invertible matrices. We conveniently recall the statement in this case.

\begin{thm}
\label{thm:cyclic-general}
	Let $M\in\GL_n(\C)$ and let $X$ be the subgroup of $\GL_n(\C)$ generated by $M$. Then $\irr(\overline{X})=|G(M)_{\mathrm{tor}}|$ and the irreducible components are pairwise isomorphic toric varieties of dimension 
	\[\dim\overline{X}=\begin{cases}
	\rk G(M) & \textup{if $M$ is diagonalizable}, \\
	\rk G(M)+1 & \textup{otherwise} .
	\end{cases}  
	\]
\end{thm}
As we will be dealing with cyclic subgroups of the form $X=\gen{M}$ with $M\in\GL_n(\C)$, throughout the present section we will make implicit use of Lemma \ref{lem:conjugation} by assuming that the matrix $M$ is given in Jordan normal form. 


We remark that the content of Theorem \ref{thm:cyclic-general} is not essentially new. Indeed, in the case of invertible matrices, one ends up working with algebraic groups: a number of algorithms for the computation of their defining polynomials are presented in \cite{deGraaf} and in many cases rely on Lie algebra techniques. Given the important role of toric varieties in modern applied mathematics, the results we present are in the language of algebraic geometry.

\subsection{The diagonalizable case}\label{sec:cyclic-diag}

For the convenience of the reader, we collect in the following remark the facts about toric varieties that we will be needing in this section.

\begin{rmk}\label{rmk:toric stuff}
Given a finite set $\AAA=\{\alpha_1,\dots,\alpha_n\}\subset\Z^r$, define the map $\Phi_\AAA:(\C^*)^r \to(\C^*)^n$ by
\begin{align*}
x=(x_1,\dots,x_r)&\mapsto (x^{\alpha_i}=x_1^{\alpha_{i1}}\cdot\ldots\cdot x_r^{\alpha_{ir}}\mid i\in\{1,\dots, n\}).
\end{align*}
The closure of the image of $\Phi_\AAA$ is the toric variety denoted by $Y_\AAA$. The dimension of $Y_\AAA$ is the rank of the free group generated by $\AAA$. In other words, if $A\in\Mat_{r\times n}(\Z)$ is the matrix whose columns are $\alpha_1\dots,\alpha_n$, then $\dim Y_\AAA=\rk A$. Moreover, the ideal of $Y_\AAA$ is generated by the binomials 
$x^\beta-x^\gamma$ whenever $\beta,\gamma\in(\Z_{\ge 0})^r$ satisfy $\beta-\gamma\in\ker_\Z(A)$. 
For these facts and more, see e.g.\ \cite[Section 1.1]{Cox}. 
\end{rmk}

\begin{es}
	\label{ex:toric}
	Let us consider $\AAA=\{(3,-1),(0,1),(1,1)\}.$ Then $\Phi_\AAA:(\C^*)^2\to(\C^*)^3$ is given by
	\begin{align*}
	(x_1,x_2)\mapsto (x_1^{3}x_2^{-1},x_2,x_1x_2).
	\end{align*}
In the notation of Remark \ref{rmk:toric stuff}, we have
	\[
	A=\begin{pmatrix}
	3 & 0& 1\\-1 &1 & 1
	\end{pmatrix}\]
	and so $Y_\AAA$ has dimension $\rk(A)=2$. Since $\ker_\Z(A)=\Z(1,4,-3)$, the toric variety $Y_\AAA$ is defined by the equation $xy^4=z^3$.
\end{es}
Let $M=\diag(a_1,\dots,a_n)\in\GL_n(\C)$. As in Notation \ref{not}, we let $X=\{M^k\mid k\in\Z\}$ be the group generated by $M$ and $G(M)=\langle a_1,\dots,a_n\rangle$ be the subgroup of $\C^*$ generated by the eigenvalues of $M$. 

\begin{pro}	\label{pro:diagonal cyclic torsionfree}
	 If $G(M)$ is torsionfree, then $\overline{X}$ is a toric variety and $\dim\overline{X}=\rk G(M)$.
\end{pro}	 
	 	\begin{proof}
		Set $r=\rk G(M)$. By hypothesis $G(M)$ is a free $\Z$-module of rank $r$. Let $c_1,\dots,c_r$ be a $\Z$-basis of $G(M)$. For every $i\in\{1,\dots,n\}$ and $j\in\{1,\dots,r\}$ there exists $\alpha_{ij}\in\Z$ such that
\begin{align*}
a_1=c_1^{\alpha_{11}}\cdot\ldots\cdot c_r^{\alpha_{1r}},\ldots,
a_n=c_1^{\alpha_{n1}}\cdot\ldots\cdot c_r^{\alpha_{nr}}.
\end{align*}
	We use this data to define the matrix
	\[
	A=\left(
	\begin{matrix}
	\alpha_{11} & \dots &\alpha_{n1}\\
	\vdots&&\vdots\\
	\alpha_{1r}& \dots& \alpha_{nr}
	\end{matrix}\right) \in\Mat_{r\times n}(\Z).\]
Let $\AAA\subset\Z^r$ be the set of lattice points corresponding to the columns of $A$ and let $Y_\AAA$ be the associated toric variety. 
By Remark \ref{rmk:toric stuff}, a set of generators of its ideal $I_{Y_\AAA}$ is given by binomials derived from a generating set of $\ker_\Z(A)$. Observe that every generator of $\ker_\Z(A)$ gives a binomial vanishing on $X$, so $I_{Y_\AAA}\subset I_X$. On the other hand, by \cite[Proposition 5]{KZ}, the ideal $I_X$ is generated by binomials with coefficients in $\graffe{0,\pm 1}$. For this reason, every generator of $I_X$ gives a relation in $G(M)$ and therefore an element of $\ker_\Z(A)$. This shows that $I_X=I_{Y_\AAA}$, so $\overline{X}=Y_\AAA$ is a toric variety. Since $\dim\overline{X}=\rk A$, in order to conclude it suffices to show that $\rk A=r$.

Up to reordering, we assume that the first $t$ columns of $A$ are a basis for the $\Z$-module spanned by all of its columns. Since $A$ has $r$ rows, we clearly have that $t\le r$. On the other hand, for every $j>t$, the $j$-th column $(\alpha_{j1}, \dots,\alpha_{jr})^\top$ is a $\Z$-linear combination of $(\alpha_{11}, \dots,\alpha_{1r})^\top,\dots,(\alpha_{t1}, \dots,\alpha_{tr})^\top$. Hence there exist $\lambda_{1j},\dots,\lambda_{tj}\in\Z$ 
 such that
\begin{align*}
\alpha_{j1}  =\lambda_{1j}\alpha_{11}+\ldots+\lambda_{tj}\alpha_{t1},\dots,
\alpha_{jr}=\lambda_{1j}\alpha_{1r}+\ldots+\lambda_{tj}\alpha_{tr}.
\end{align*}
This means that
\begin{align*}
a_j&=c_1^{\alpha_{j1}}\cdot\ldots\cdot c_r^{\alpha_{jr}}=c_1^{\lambda_{1j}\alpha_{11}+\ldots+\lambda_{tj}\alpha_{t1}}\cdot\ldots\cdot c_r^{\lambda_{1j}\alpha_{1r}+\ldots+\lambda_{tj}\alpha_{tr}}\\
&=c_1^{\lambda_{1j}\alpha_{11}}\cdot\ldots\cdot c_r^{\lambda_{1j}\alpha_{1r}}\cdot\ldots\cdot c_1^{\lambda_{tj}\alpha_{t1}}\cdot\ldots\cdot c_r^{\lambda_{tj}\alpha_{tr}}\\&=a_1^{\lambda_{1j}}\cdot\ldots\cdot a_t^{\lambda_{tj}}.
\end{align*}
Therefore $a_{t+1},\dots,a_n\in\langle a_1,\dots,a_t\rangle$ and so we conclude that $t\ge r$.\end{proof}

The structure of diagonalizable algebraic groups is discussed in \cite[Section 3.9]{deGraaf}. In particular, Proposition 3.9.7 ensures that a diagonalizable algebraic subgroup of $\GL_n(\C)$ splits as a direct product of a finite group and an $r$-dimensional torus, where $r$ is the rank of its associated lattice (in the language of \cite{Cox}, the lattice associated to the toric variety). The arguments we use in the proof of Proposition \ref{pro:diagonal cyclic torsionfree} resemble those from \cite[Section 3.3]{DJK} or \cite[Section 3.9]{deGraaf} though in a slightly different language.

In his PhD Thesis (University of Leipzig, 2020), G\"orlach presents a reformulation of \cite[Proposition 3.9.7]{deGraaf} from the point of view of Hadamard product of algebraic varieties.


\begin{pro}\label{pro:diagonal cyclic} 	
The variety $\overline{X}$ has $|G(M)_{\mathrm{tor}}|$ irreducible components. The components are pairwise isomorphic toric varieties of dimension $\rk G(M)$.
\end{pro}

\begin{proof}
Set $q=|G(M)_{\mathrm{tor}}|$. For every $i\in\{0,\dots,q-1\}$, define $Y_i=\{M^{kq+i}\mid k\in\Z\}$. Then $X$ is the disjoint union of the $Y_i$'s and
\[
\overline{X}=\overline{Y_0\cup\ldots\cup Y_{q-1}}=\overline{Y_0}\cup\ldots\cup \overline{Y_{q-1}}.
\]
Observe that $Y_i=\{M^i\cdot (M^q)^k\mid k\in\Z\}$ equals the image of $Y_0=\{(M^q)^k\mid k\in\Z\}$ under a linear automorphism of $\Mat_n(\C)$, namely multiplication by $M^i$. Moreover we have
		\[M^q=\diag(a_1^q,\dots,a_n^q).\]
		By construction, the group $\langle a_1^q,\dots,a_n^q\rangle$ is torsionfree of rank equal to $\rk G(M)$. Proposition \ref{pro:diagonal cyclic torsionfree} yields that $\overline{Y_i}$ has dimension $\rk G(M)$ and, being toric, $\overline{Y_i}$ is irreducible.
\end{proof}

We remark  that, in the induced topology, the connected components of $\overline{X} \cap \GL_n(\C)$ are precisely the intersections $\overline{Y_i}\cap \GL_n(\C)$, where $\overline{Y_i}$ is as in the proof of Proposition \ref{pro:diagonal cyclic}. In particular, $\overline{Y_0}\cap \GL_n(\C)$ is the unique irreducible component of $\overline{X}\cap\GL_n(\C)$ that contains the identity matrix.  For more on connectedness, see for example \cite[Section 3.2]{deGraaf}.

With the next example, we would like to hint to how much information toric geometry carries. 
We apply results from \cite[Chapter 2.4]{Cox} to check whether $\ov{X}$ is normal and to compute its singular locus. Moreover, we apply \cite[Theorem 13.4.1]{Cox} to compute the degree of $\ov{X}$. Recall that the \emph{normalized volume} of a polytope $P\subseteq\R^n$, denoted by $\vol P$, is $n!$ times its Lebesgue measure.

\begin{es}\label{es:polytope}
	Let $M=\diag(1,2,3,4)$ and let $X$ be the group generated by $M$. The eigenvalues of $M$ generate $G(M)=\langle 2,3\rangle\cong\Z^2$
	. 
	Following the proof of Proposition \ref{pro:diagonal cyclic torsionfree}, we have
	\[
	A=\begin{pmatrix}
	0 & 1 & 0 & 2\\
	0 & 0 & 1 & 0
	\end{pmatrix}.
	\]
	In this case $\ker_\Z A=\langle(1,0,0,0),(0,2,0,-1)\rangle$. 
Viewed as a subvariety of $\p^3$, the toric variety $\ov{X}$ is defined by $x_1^2=x_0x_3$ and corresponds to the polytope
	
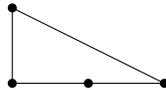
\[\begin{tikzpicture}
	\draw (0,0) -- (2,0) -- (0,1) -- cycle;
	\fill[black] (0,0) circle (0.06cm);
	\fill[black] (1,0) circle (0.06cm);
	\fill[black] (2,0) circle (0.06cm);
	\fill[black] (0,1) circle (0.06cm);
	\end{tikzpicture}\]
which we denote by $P$. 
 Such polytope is normal of dimension 2, so the projective variety $\ov{X}$ is a normal surface. However, $P$ is not smooth, so $\ov{X}$ is singular. More precisely, its singular locus is a point. The degree of $\ov{X}$ is $\vol P=2$. As shown in \cite[Example 2.4.6]{Cox}, the variety $\ov{X}$ is the weighted projective space $\p(1,1,2)$ embedded as a quadric cone in $\p^3$.
\end{es} 
The next result shows that we can realize every toric variety as the Zariski closure of a cyclic subgroup of $\GL_n(\C)$.


\begin{pro}
	\label{pro:inverse toric problem}
	Let $Y\subseteq\C^n$ be an affine toric variety and identify $\C^n$ with the space of diagonal matrices. 
	Then there exists a diagonal matrix $M\in\GL_n(\C)$ such that $Y=\ov{\Span{M}}$ in $\C^n$.
	\end{pro}
	
	\begin{proof}
		Let $\{\alpha_1,\dots,\alpha_n\}\subseteq\Z^r$ be a set of lattice points defining $Y$ as a toric variety. 
		Let
		\[
		A=\left(
		\begin{matrix}
		\alpha_{11} & \dots &\alpha_{n1}\\
		\vdots&&\vdots\\
		\alpha_{1r}& \dots& \alpha_{nr}
		\end{matrix}\right) \in\Mat_{r\times n}(\Z)\]
		be the matrix with columns $\alpha_1,\dots,\alpha_n$. Let $c_1,\dots,c_r$ be $r$ distinct prime numbers and set
		\begin{align*}
		a_1=c_1^{\alpha_{11}}\cdot\ldots\cdot c_r^{\alpha_{1r}},\dots,
		a_n=c_1^{\alpha_{n1}}\cdot\ldots\cdot c_r^{\alpha_{nr}}.
		\end{align*}
By defining $M=\diag(a_1,\dots,a_n)$ and following the proof of Proposition \ref{pro:diagonal cyclic torsionfree}  backwards, we find $Y=\ov{\Span{M}}$.
	\end{proof}

An immediate consequence of Proposition \ref{pro:inverse toric problem} is that we can cook up cyclic matrix groups whose closure has arbitrary dimension, degree and number of irriducible components. However, we observe that, in contrast to the case of toric varieties, not all binomial varieties can be realized as closures of cyclic subgroups of $\GL_n(\C)$. 


The next example shows a way of applying Proposition \ref{pro:diagonal cyclic} in a simple non-cyclic setting.

\begin{es}
	Define $X=\gen{A,B}$ where 
	\[
	A=\begin{pmatrix}
	2 & 0 \\
	0 & 1
	\end{pmatrix} \textup{ and } 
	B=\begin{pmatrix}
	1 & 0 \\
	0 & 2
	\end{pmatrix}. 
	\]
Then, for any $d\in\Z$, the group $X$ contains the cyclic subgroup
	\[
	 Y_d=\gen{AB^d}=\graffe{
		\, \begin{pmatrix}
		2^h & 0 \\
		0 & 2^{hd}
		\end{pmatrix}
		\mid h\in\Z
	}.
	\]
	Thanks to Proposition \ref{pro:diagonal cyclic torsionfree}, the closure of each $Y_d$ is a curve and so, $\overline{X}$ containing infinitely many curves, the dimension of $\overline{X}$ is 2. In particular, $\ov{X}$ is a plane in $\C^4$.
\end{es}

\subsection{The unipotent case}\label{sec:unipotent}

In this section we 
consider the case of unipotent matrices and prove Theorem \ref{thm:cyclic-general}. Let $M\in\GL_n(\C)$ and let $X$ be the subgroup of $\GL_n(\C)$ generated by $M$. 
Without loss of generality, we assume that $M$ is in Jordan normal form. Let $M_s$ and $M_u$ be respectively the semisimple and the unipotent part of $M$, which satisfy $M_sM_u=M_uM_s$. In particular, $M=M_sM_u$ is upper triangular, $M_s$ is diagonal and $M_u$ is upper unitriangular.
We remark that the eigenvalues of $M$ are the same as the eigenvalues of $M_s$. 
We define additionally $X_s=\overline{\{M_s^k \mid k\in\Z\}}$ and $X_u=\overline{\{M_u^k \mid k\in\Z\}}$.

The proof of the following result is an easy computation.


\begin{lem}\label{lem:jordanblocks}
Let $\lambda\in\C^*$, $k\in\Z_{\geq 0}$, and let $J(m,\lambda)=(b_{ij})\in\GL_m(\C)$ be defined by
\[
b_{ij}=\begin{cases}
1 & \textup{if } i=j, \\
\lambda & \textup{if } j=i+1, \\
0 & \textup{otherwise}.
\end{cases}
\]
Write $J(m,\lambda)^k=(a_{ij})$. Then 
\begin{equation}\label{eq:aij}
a_{ij}=\begin{cases}
0 & \textup{if } i>j, \\
\binom{k}{j-i}\lambda^{j-i} & \textup{otherwise}
\end{cases}
\end{equation}
and, for each $r\in\graffe{1,\ldots, m-1}$, the following holds:
\begin{equation}\label{eq:factorial}
r! a_{1,r+1}=\prod_{i=0}^{r-1}(a_{12}-i\lambda).
\end{equation}
\end{lem}


\begin{lem}\label{lem:unipotent}
Assume that $M_u\neq 1$ and let $m$ be the biggest size of a Jordan block of $M$. Then $X_u$ is a degree $m-1$ rational normal curve. 
\end{lem}

\begin{proof}
	Let $d$ denote the number of Jordan blocks of $M$, arbitrarily ordered. For each $l\in\{1,\dots,d\}$, let $\lambda_l$ and $m(l)$ denote respectively the eigenvalue and size corresponding to the $l$-th Jordan block of $M$.
Set $J_l=J(m(l),\lambda_l^{-1})$ so that, for every $k\in\Z$, we have
$M_u^k=\diag(J_1^k,\dots,J_d^k)$.
Fix now $k\in\Z$ and write $a_{l, ij}$ for the $(i,j)$-th entry of $J_l^k$. By Lemma \ref{lem:jordanblocks}(1), all entries of $J_l^k$ are linear functions of entries in the first row of $J_l^k$ and thus, by Lemma \ref{lem:jordanblocks}(2), polynomials in $a_{l,12}$. Furthermore, by Lemma \ref{lem:jordanblocks}(1), the blocks $J_l^k$ and $J_s^k$ are compared via
	\[a_{l,12}=k\lambda_l^{-1}=\frac{\lambda_s}{\lambda_l}\cdot k\lambda_s^{-1}=\frac{\lambda_s}{\lambda_l}\cdot a_{s,12}.\]
Fix $J\in\graffe{J_1,\ldots,J_d}$ to be an element of maximal size $m$.
	Then $X_u$ is contained in a linear space $L$ of dimension $m$, with coordinates $x_1,\dots,x_m$ corresponding to the entries of the first row of $J$. More precisely, $X_u$ is contained in the image of the map $\tilde{f}:\C\to L$ defined by
	\begin{align*}
	t\mapsto \left( 1,t\lambda^{-1},\frac{t(t-1)}{2}\lambda^{-2},\dots,\frac{1}{(m-1)!}\prod_{j=0}^{m-2}(t-j)\lambda^{-m+1}\right).
	\end{align*}
	Since the image of $\tilde{f}$ is an irreducible curve and $X_u$ is infinite, $\tilde{f}(\C)=X_u$. After applying the first linear change of coordinates
	\[(x_1,x_2,x_3,\dots,x_m)\mapsto (x_1,\lambda x_2,2\lambda^2x_3,\dots,(m-1)!\lambda^{m-1}x_m),
	\] $X_u$ is parametrized by the map $f:\C\to L$ defined by
\[	t\mapsto\left( 1,t,t(t-1),\dots,\prod_{j=0}^{m-2}(t-j)\right).\]
To show that $X_u$ is a degree $m-1$ rational normal curve, 
we recursively construct linear polynomials $l_1(x_1),l_2(x_1,x_2),\ldots,l_m(x_1,\dots,x_m)$ such that, for each $r\in\{1,\dots,m\}$,  the map $\phi_r:\C^m\to\C^m$ defined by
\[\phi_r(x_1,\dots,x_m)=(l_1(x_1),\dots,l_r(x_1,\dots,x_r),x_{r+1},\dots,x_m)\]
has the property that the first $r$ entries of $f\circ\phi_r(x_1,\dots,x_m)$ equal $(1,t,t^2,\dots,t^{r-1})$. 
Set $l_1(x_1)=x_1$ and $l_2(x_1,x_2)=x_2$. Assume now that $l_1\dots,l_r$ are given and let us define $l_{r+1}$. By the induction hypothesis, the change of coordinates
$(x_1,\dots,x_m)\mapsto (l_1(x_1),\dots,l_r(x_1,\dots,x_r),x_{r+1},\dots,x_m)$
turns $f$ into
\[t\mapsto \left( 1,t,t^2,\dots,t^{r-1},\prod_{j=0}^{r-1}(t-j),\dots,\prod_{j=0}^{m-2}(t-j)\right).
\]
Now the $(r+1)$-th entry is of the form $t^r+c_{r-1}t^{r-1}+\ldots+c_1t+c_0$ for some $c_0,\ldots,c_{r-1}\in\C$. We conclude by defining \[l_{r+1}(x_1,\ldots,x_{r+1})=x_{r+1}-c_{r-1}x_r-\ldots-c_1x_2-c_0x_1, \]
which is linear in $x_1,\ldots, x_{r+1}$ and satisfies by construction the required inductive property.
\end{proof}

Lemma \ref{lem:unipotent} is a different instance of \cite[Proposition 4.3.10]{deGraaf} for algebraic subgroups of $\GL_n(\C)$, though our proof does not rely on Lie theory. Moreover, as a consequence of \cite[Corollary 4.3.11]{deGraaf}, the Zariski closure of any unipotent subgroup is (connected and thus) irreducible in $\Mat_n(\C)$. For more about unipotent algebraic groups in this context, see for example \cite[Section 4.3.2]{deGraaf}.


\begin{es}\label{es:twisted cubic via unitriangular}
We use the notation from Lemma \ref{lem:jordanblocks}. Define
\[
M=
\begin{pmatrix}
 1/5 & 1 & 0 & 0 \\
 0 & 1/5 & 1 & 0 \\
 0 & 0 & 1/5 & 1 \\
 0 & 0 & 0 & 1/5
\end{pmatrix},
\]
which is already in Jordan normal form. In this case 
\[
M_s=\diag(
1/5 , 1/5,1/5,1/5) \textup{ and }
M_u=J(4,5).
\]
By Lemma \ref{lem:jordanblocks}(1), for each $k\in\Z$ one has 
\[
M_u^k =J(4,5)^k
=
\begin{pmatrix}
 1 & 5k & 5^2\cdot\frac{k(k-1)}{2} & 5^3\cdot\frac{k(k-1)(k-2)}{6} \\
 0 & 1 & 5k & 5^2\cdot\frac{k(k-1)}{2} \\
 0 & 0 & 1 & 5k \\
 0 & 0 & 0 & 1
\end{pmatrix}.
\]
Denoting by $x_{ij}$ the $16$ independent variables corresponding to the entries of a matrix in $\Mat_4(\C)$, we see that $X_u$ is contained in the $4$-dimensional linear space $L$ defined by the equations
\begin{align*}
x_{ij}=0\textup{ for } i<j, & \quad x_{13}=x_{24},\\
x_{ii}=1\textup{ for }i\in\graffe{1,\ldots, 4}, & \quad  
x_{12}=x_{23}=x_{34}.
\end{align*}
We identify $L$ with the affine space $\C^4$, with coordinates $x_1,x_2,x_3,x_4$ corresponding to the entries $x_{11},x_{12},x_{13},x_{14}$ of the first row of $M_u^k$. Then $X_u$ is the image of the map $\C\to L$ defined by $$t\mapsto \left( 1,5t,\frac{25t(t-1)}{2},\frac{125t(t-1)(t-2)}{6}\right).$$ After the changes of coordinates
\[(x_1,x_2,x_3,x_4)\mapsto \left( x_1,\frac{x_2}{5},\frac{2x_3}{25}+\frac{x_2}{5},\frac{6x_4}{125}+\frac{6x_3}{25}+\frac{x_2}{5}\right)
,\]
we see that $X_u$ is the image of $t\mapsto (1,t,t^2,t^3)$, so $X_u$ is the twisted cubic curve in the hyperplane defined by $x_1=1$ in $L$.
\end{es}

\begin{pro}\label{pro:sum dims} 
The following equalities hold:
\[
\dim\overline{X}=\dim X_s + \dim X_u  \textup{ and } 
\irr(\overline{X})=\irr(X_s).
\] 
\end{pro}

\begin{proof}
By Remark \ref{rmk:dim-irr}, the dimension and the number of irreducible components of $\overline{X}$ remain invariant when intersecting $\overline{X}$ with $\GL_n(\C)$. For \emph{this proof only}, we will write $\overline{X}$ to mean the Zariski closure of $X$ in $\GL_n(\C)$. This applies also to $X_s$ and $X_u$. Recall that $\overline{X}$, $X_s$, and $X_u$ are in this case subgroups of $\GL_n(\C)$, see for example \cite[Lemma 2.2.4]{Springer}.

We start by observing that $\ov{X}$ is abelian. 
Indeed, the commutator map $\ov{X}\times \ov{X}\rightarrow \GL_n(\C)$ is continuous and trivial on the dense subset $X\times X$, therefore it is itself trivial.
Now, the group $\ov{X}$ being abelian, \cite[Theorem 15.5]{Humphreys} yields that $\overline{X}\cong X_s\times X_u$
. In particular, we get $\dim\overline{X}=\dim X_s+\dim X_u$. Lemma \ref{lem:unipotent} ensures that $X_u$ is irreducible and thus we also have that $\irr(\overline{X})=\irr(X_s)$.
\end{proof}

We prove here Theorem \ref{thm:cyclic-general}. 
From Proposition \ref{pro:sum dims} we know that $\dim\overline{X}=\dim X_s+\dim X_u$ and $\irr(\overline{X})=\irr(X_s)$. By Proposition \ref{pro:diagonal cyclic}, we have $\irr(\overline{X})=|G(M)_{\mathrm{tor}}|$ and, combined with Lemma \ref{lem:unipotent}, that 
$$\dim\overline{X}=\rk G(M) +\dim X_u
=\begin{cases}
\rk G(M) & \textup{if } M_u=1,\\
\rk G(M) +1 & \textup{otherwise.}
\end{cases}
$$
In conclusion, as a consequence of Proposition \ref{pro:diagonal cyclic} and the fact that $\overline{X}\cap \GL_n(\C)$ is dense in $\overline{X}$, the irreducible components of $\overline{X}$ are toric varieties.


\begin{cor}\label{cor:potenze hanno la stessa chiusura} 
Let $q\in\Z$. If $G(M)$ is torsionfree, then $\ov{X}=\ov{\Span{M^q}}$.
\end{cor}
	\begin{proof}
		Let $a_1,\dots,a_n$ be the eigenvalues of $M$ and assume that $G(M)$ is torsionfree. Then the eigenvalues of $M^q$ are $a_1^q,\dots,a_n^q$ and $\langle a_1^q,\dots,a_n^q\rangle$ is a free $\Z$-submodule of $G(M)$ of the same rank as $G(M)$. By Theorem \ref{thm:cyclic-general}, the varieties $\ov{\Span{M}}$ and $\ov{\Span{M^q}}$ are both irreducible of the same dimension. Since $\ov{\Span{M}}\supseteq\ov{\Span{M^q}}$, they are the same.
	\end{proof}

\section{Zariski closure of a cyclic semigroup}\label{sec:semigroups}
The purpose of this section is to prove Theorem \ref{thm:semigroups}. We start with an example to illustrate the argument we will use in the proof.

\begin{es}
Let $M\in\Mat_n(\C)$ be defined by
\[
M=\begin{pmatrix}
0 & 1 & 0 \\
0 & 0 & 0 \\
0 & 0 & 2
\end{pmatrix}
\]
and let $X=\{M^k \mid k\in\Z_{>0}\}$. Then $M^2=\diag(0,0,4)$ and thus we have 
\[
X=\graffe{M}\ \dot{\cup}\  \{\diag(0,0,2^k) \mid k\geq 2\}. 
\]
We observe that the set $\{\diag(0,0,2^k) \mid k\geq 2\}$ consists of infinitely many collinear points. In particular, we get that 
\[
\overline{X}=\graffe{M}\ \dot{\cup}\  \graffe{\diag(0,0,z) \mid z\in \C} 
\]
and so $X$ is the disjoint union of a point and a line.
\end{es}

Until the end of this section, we will work under the hypotheses of Theorem \ref{thm:semigroups}. We proceed by considering disjoint cases. 

Assume first that $\mc{E}(M)=\vu $.
In this case the only eigenvalue of $M$ is $0$, which implies that $M$ is nilpotent
. 
Since $M\neq 0$ by hypothesis, $M^{\nu}$ is the smallest power of $M$ that is equal to $0$ and so $\overline{X}$ consists of $\nu$ points. To conclude, define $X_0=\overline{X}$ and $X_1=\vu $.

Assume now that $M$ is invertible, so $\nu=0$ and $\mc{E}(M)\neq\vu $. Define $X_0=\vu $ and $X_1=\overline{X}$. We are now done thanks to Theorem \ref{thm:cyclic-general}.

To conclude, assume that $M$ is not invertible and $\mc{E}(M)\neq\vu $. In this case, $\nu\geq 1$ and there exist positive integers $m$ and $p$ and matrices $N\in\Mat_m(\C)$ strictly upper triangular and $M_1\in\GL_p(\C)$ upper triangular such that $M$ has the following block shape:
\[
M=\begin{pmatrix}
N & 0 \\
0 & M_1
\end{pmatrix}.
\] 
Fix such matrices $N$ and $M_1$. Then $N$ is nilpotent and $\nu$ is the smallest exponent annihilating $N$. It follows that
\[
X = \{M^k \mid k\in\{ 1,\ldots,\nu-1\}\}\ \dot{\cup}\  \graffe{\begin{pmatrix}
0 & 0 \\
0 & M_1^k
\end{pmatrix} \mid k\geq \nu }.
\]
Write $X_0=\{M^k \mid k\in\{ 1,\ldots,\nu-1\}\}$ and 
$$Y_1=\graffe{\begin{pmatrix}
0 & 0 \\
0 & M_1^k
\end{pmatrix} \mid k\geq \nu }.$$
Then $X_0$ is a closed variety consisting of $\nu-1$ points. Set $X_1=\overline{Y_1}$. We observe that the semigroup generated by $M_1$ in $\GL_p(\C)$ is the image of $Y_1$ under a linear automorphism of $\Mat_n(\C)$. It follows from Proposition \ref{pro:semi vs gp} that $X_1$ is isomorphic to the Zariski closure of $\gen{M_1}$ in $\Mat_p(\C)$. Thanks to Theorem \ref{thm:cyclic-general}, the proof of Theorem \ref{thm:semigroups} is now complete.

\section{Computation of closures of matrix groups}
We conclude the paper with a sinthetic discussion of the available  algorithms for the computation of Zariski closures of matrix groups.

We start by remarking that, while $\C$ is the most convenient field choice for algebraic geometry, this is certainly not the case for computer algebra softwares.  For computational purposes, it is indeed necessary to work over a field that is suitable for symbolic computations, for instance the field of rational numbers. 

Ideally, one wishes for an algorithm that takes as input a
list of matrices $M_1,\dots,M_t\in\Mat_n(\mathbb{Q})$ and returns as output the ideal of the Zariski closure of the group or semigroup generated by $M_1,\dots,M_t$. Such an algorithm would provide the strongest polynomial invariants of $\gen{M_1,\ldots, M_t}$.


When all the matrices are invertible, it makes sense to consider the group they generate: algorithms computing the closure of such group in $\GL_n(\C)$ are presented in \cite[Section 3]{DJK} and in \cite[Chapter 4.6]{deGraaf}. The computation of the closure in $\Mat_n(\C)$ or $\Mat_n(\R)$ of the generated semigroup is addressed in \cite{HOPW/18}.  For a number of related problems, see for example \cite{Babaibeals, JKK, KZ, OuWo, SWZ}.

Some of these results concern \textit{decidability}, i.e.\ the existence of an algorithmic solution. Among the implementations we mention \cite[Algorithm 3]{KZ}, implemented in \textsc{ Mathematica 5} \cite{M5}, and various algorithms presented in \cite{deGraaf} and implemented in \textsc{GAP4} \cite{gap4}, \textsc{Magma} \cite{magma}, and \textsc{Singular} \cite{singular}. 

To the best of our knowledge, no complexity analysis has been run in \cite{DJK, HOPW/18, KZ}.  In the preface to \cite{deGraaf}, the author writes: \textquotedblleft We do not consider the complexity of algorithms as they very often are bad.
Indeed, quite a few algorithms use Gr\"obner bases, and the complexity of the
algorithms to compute the latter is known to be doubly exponential\textquotedblright. It is however worth mentioning that  many of these algorithms rely on a polynomial-time algorithm of Ge \cite[Theorem 1.1]{Ge}, dealing with units in number fields.  The last result is generalized in \cite[Theorem 1.11]{Lenstra} to arbitrary $\mathbb{Q}$-algebras.

\end{document}